\documentclass{amsart}
\usepackage{amsmath}
\usepackage{amsfonts}

\newtheorem{theorem}{Theorem}
\theoremstyle{plain}

\newtheorem{corollary}{Corollary}

\newtheorem{definition}{Definition}

\newtheorem{lemma}{Lemma}

\newtheorem{remark}{Remark}

\numberwithin{equation}{section}
\input{tcilatex}

\begin{document}
\title[New Inequalities for $h-$convex Functions]{On New Inequalities for $%
h- $convex Functions via Riemann-Liouville Fractional Integration}
\author{Mevl\"{u}t TUN\c{C}}
\address{University of Kilis 7 Aral\i k, Faculty of Science and Arts,
Department of Mathematics, 79000, Kilis, Turkey}
\email{mevluttunc@kilis.edu.tr}

\begin{abstract}
In this paper, some new inequalities of the Hermite-Hadamard type for $h-$%
convex functions via Riemann-Liouville fractional integral are given.
\end{abstract}

\maketitle

\section{INTRODUCTION}

Let $f:I\subseteq 
\mathbb{R}
\rightarrow 
\mathbb{R}
$ be a convex function and let $a,b\in I,$ with $a<b.$ The following
inequality;%
\begin{equation}
\ \ \ f\left( \frac{a+b}{2}\right) \leq \frac{1}{b-a}\int_{a}^{b}f(x)dx\leq 
\frac{f(a)+f(b)}{2}  \label{a}
\end{equation}%
is known in the literature as Hadamard's inequality. Both inequalities hold
in the reversed direction if $f$ is concave.

In \cite{varr}, Varo\v{s}anec introduced the following class of functions.

\begin{definition}
Let $h:J\subset 
\mathbb{R}
\rightarrow 
\mathbb{R}
$ be a positive function. We say that $f:I\subset 
\mathbb{R}
\rightarrow 
\mathbb{R}
$ is $h-$convex function or that $f$ belongs to the class $SX(h,I)$, if $f$
is nonnegative and for all $x,y\in I$ and $\lambda \in (0,1)$ we have%
\begin{equation}
f(\lambda x+(1-\lambda )y)\leq h(\lambda )f(x)+h(1-\lambda )f(y).  \label{9}
\end{equation}
\end{definition}

If the inequality in (\ref{9}) is reversed, then $f$ is said to be $h-$%
concave, i.e., $f\in SV(h,I).$

Obviously, if $h\left( \lambda \right) =\lambda $, then all nonnegative
convex functions belong to $SX\left( h,I\right) $\ and all nonnegative
concave functions belong to $SV(h,I);$ if $h(\lambda )=\frac{1}{\lambda },$
then $SX(h,I)=Q(I);$ if $h(\lambda )=1,$ then $SX(h,I)\supseteq P(I)$ and if 
$h(\lambda )=\lambda ^{s},$ where $s\in \left( 0,1\right) ,$ then $%
SX(h,I)\supseteq K_{s}^{2}.$ For some recent results for $h-$convex
functions we refer to the interested reader to the papers \cite{zeki}-\cite%
{BU}.

\begin{definition}
\lbrack See \cite{HA}] A function $h:J\rightarrow 
\mathbb{R}
$ is said to be a superadditive function if%
\begin{equation}
h(x+y)\geq h(x)+h(y)  \label{1.7}
\end{equation}%
for all $x,y\in J$.
\end{definition}

In \cite{zeki}, Sar\i kaya et al. proved the following Hadamard type
inequalities for $h-$convex functions.

\begin{theorem}
Let $f\in SX(h,I),$ $a,b\in I$ with $a<b$ and $f\in L_{1}[a,b].$ Then%
\begin{equation}
\frac{1}{2h\left( \frac{1}{2}\right) }f\left( \frac{a+b}{2}\right) \leq 
\frac{1}{b-a}\int_{a}^{b}f(x)dx\leq \lbrack f(a)+f(b)]\int_{0}^{1}h(\alpha
)d\alpha .  \label{10}
\end{equation}
\end{theorem}

In \cite{zeki2}, Sar\i kaya et al. proved the following Hadamard type
inequalities for fractional integrals as follows.

\begin{theorem}
\label{diz} Let $f:[a,b]\rightarrow 
\mathbb{R}
$ be positive function with $0\leq a<b$ and $f\in L_{1}[a,b].$ If $f$ is
convex function on $[a,b]$, then the following inequalities for fractional
integrals hold:%
\begin{equation}
f\left( \frac{a+b}{2}\right) \leq \frac{\Gamma (\alpha +1)}{2(b-a)^{\alpha }}%
\left[ J_{a^{+}}^{\alpha }(b)+J_{b^{-}}^{\alpha }(a)\right] \leq \frac{%
f(a)+f(b)}{2}  \label{16}
\end{equation}%
with $\alpha >0.$
\end{theorem}

Now we give some necessary definitions and mathematical preliminaries of
fractional calculus theory which are used throughout this paper.

\begin{definition}
Let $f\in L_{1}[a,b].$ The Riemann-Liouville integrals $J_{a^{+}}^{\alpha }f$
and $J_{b^{-}}^{\alpha }f$ of order $\alpha >0$ with $a\geq 0$ are defined by%
\begin{equation*}
J_{a^{+}}^{\alpha }f(x)=\frac{1}{\Gamma (\alpha )}\underset{a}{\overset{x}{%
\int }}\left( x-t\right) ^{\alpha -1}f(t)dt,\text{ \ }x>a
\end{equation*}%
and%
\begin{equation*}
J_{b^{-}}^{\alpha }f(x)=\frac{1}{\Gamma (\alpha )}\underset{x}{\overset{b}{%
\int }}\left( t-x\right) ^{\alpha -1}f(t)dt,\text{ \ }x<b
\end{equation*}%
respectively where $\Gamma (\alpha )=\underset{0}{\overset{\infty }{\int }}%
e^{-u}u^{\alpha -1}du.$ Here is $J_{a^{+}}^{0}f(x)=J_{b^{-}}^{0}f(x)=f(x).$
\end{definition}

In the case of $\alpha =1$, the fractional integral reduces to the classical
integral.

For some recent results connected with \ fractional integral inequalities
see \cite{anastas}-\cite{dahtab} and \cite{zeki2}.

In \cite{zeki2}, Sarikaya \textit{et al.} proved a variant of the identity
is established by Dragomir and Agarwal in \cite[Lemma 2.1]{DA} for
fractional integrals as the following.

\begin{lemma}
\label{l1} Let $f:[a,b]\rightarrow 
\mathbb{R}
$ be a differentiable mapping on $(a,b)$ with $a<b.$ If $f^{\prime }\in
L[a,b],$ then the following equality for fractional integrals holds:%
\begin{eqnarray*}
&&\frac{f(a)+f(b)}{2}-\frac{\Gamma (\alpha +1)}{2(b-a)^{\alpha }}\left[
J_{a^{+}}^{\alpha }f(b)+J_{b^{-}}^{\alpha }f(a)\right] \\
&=&\frac{b-a}{2}\int_{0}^{1}\left[ \left( 1-t\right) ^{\alpha }-t^{\alpha }%
\right] f^{\prime }(ta+(1-t)b)dt.
\end{eqnarray*}
\end{lemma}

\bigskip The aim of this paper is to establish Hadamard type inequalities
for $h-$convex functions via Riemann-Liouville fractional integral.

\section{MAIN RESULTS}

\begin{theorem}
\label{metu} Let $f\in SX(h,I),$ $a,b\in I$ with $a<b$ and $f\in L_{1}[a,b].$
Then one has inequality for $h-$convex functions via fractional integrals%
\begin{eqnarray}
&&\frac{\Gamma (\alpha )}{(b-a)^{\alpha }}\left[ J_{a^{+}}^{\alpha
}(b)+J_{b^{-}}^{\alpha }(a)\right]  \label{11} \\
&\leq &\left[ f(a)+f(b)\right] \int_{0}^{1}t^{\alpha -1}\left[ h(t)+h(1-t)%
\right] dt  \notag \\
&\leq &\frac{2\left[ f(a)+f(b)\right] }{\left( \alpha p-p+1\right) ^{\frac{1%
}{p}}}\left( \int_{0}^{1}\left( h\left( t\right) \right) ^{q}dt\right) ^{%
\frac{1}{q}}  \notag
\end{eqnarray}%
where $p^{-1}+q^{-1}=1.$
\end{theorem}

\begin{proof}
Since $f\in SX(h,I)$, we have%
\begin{equation*}
f(tx+(1-t)y)\leq h(t)f(x)+h(1-t)f(y)
\end{equation*}%
and%
\begin{equation*}
f((1-t)x+ty)\leq h(1-t)f(x)+h(t)f(y).
\end{equation*}%
By adding these inequalities we get%
\begin{equation}
f(tx+(1-t)y)+f((1-t)x+ty)\leq \left[ h(t)+h(1-t)\right] \left[ f(x)+f(y)%
\right] .  \label{12}
\end{equation}

By using (\ref{12}) with $x=a$ and $y=b$ we have%
\begin{equation}
f(ta+(1-t)b)+f((1-t)a+tb)\leq \left[ h(t)+h(1-t)\right] \left[ f(a)+f(b)%
\right] .  \label{13}
\end{equation}%
Then multiplying both sides of (\ref{13}) by $t^{\alpha -1}$ and integrating
the resulting inequality with respect to $t$ over $[0,1]$, we get%
\begin{eqnarray*}
&&\int_{0}^{1}t^{\alpha -1}\left[ f(ta+(1-t)b)+f((1-t)a+tb)\right] dt \\
&\leq &\int_{0}^{1}t^{\alpha -1}\left[ h(t)+h(1-t)\right] \left[ f(a)+f(b)%
\right] dt,
\end{eqnarray*}%
\begin{eqnarray}
&&\frac{\Gamma (\alpha )}{(b-a)^{\alpha }}\left[ J_{a^{+}}^{\alpha
}(b)+J_{b^{-}}^{\alpha }(a)\right]  \label{14} \\
&\leq &\left[ f(a)+f(b)\right] \int_{0}^{1}t^{\alpha -1}\left[ h(t)+h(1-t)%
\right] dt  \notag
\end{eqnarray}%
and thus the first inequality is proved.

To obtain the second inequality in (\ref{11}), by using H\"{o}lder
inequality for the right hand side of (\ref{14}), we obtain 
\begin{eqnarray*}
&&\int_{0}^{1}t^{\alpha -1}\left[ h\left( t\right) +h\left( 1-t\right) %
\right] dt \\
&\leq &\left( \int_{0}^{1}\left( t^{\alpha -1}\right) ^{p}dt\right) ^{\frac{1%
}{p}}\left( \int_{0}^{1}\left( h\left( t\right) +h\left( 1-t\right) \right)
^{q}dt\right) ^{\frac{1}{q}} \\
&=&\left( \left. \frac{t^{\alpha p-p+1}}{^{\alpha p-p+1}}\right\vert
_{0}^{1}\right) ^{^{\frac{1}{p}}}\left( \int_{0}^{1}\left( h\left( t\right)
+h\left( 1-t\right) \right) ^{q}dt\right) ^{\frac{1}{q}} \\
&=&\left( \frac{1}{^{\alpha p-p+1}}\right) ^{^{\frac{1}{p}}}\left(
\int_{0}^{1}\left( h\left( t\right) +h\left( 1-t\right) \right)
^{q}dt\right) ^{\frac{1}{q}}
\end{eqnarray*}%
Then using Minkowski inequality%
\begin{eqnarray*}
&&\left( \frac{1}{^{\alpha p-p+1}}\right) ^{^{\frac{1}{p}}}\left(
\int_{0}^{1}\left( h\left( t\right) +h\left( 1-t\right) \right)
^{q}dt\right) ^{\frac{1}{q}} \\
&\leq &\left( \frac{1}{^{\alpha p-p+1}}\right) ^{^{\frac{1}{p}}}\left[
\left( \int_{0}^{1}\left( h\left( t\right) \right) ^{q}dt\right) ^{\frac{1}{q%
}}+\left( \int_{0}^{1}\left( h\left( 1-t\right) \right) ^{q}dt\right) ^{%
\frac{1}{q}}\right] \\
&=&\frac{2}{\left( ^{\alpha p-p+1}\right) ^{^{\frac{1}{p}}}}\left(
\int_{0}^{1}\left( h\left( t\right) \right) ^{q}dt\right) ^{\frac{1}{q}}
\end{eqnarray*}%
where the proof is completed.
\end{proof}

\begin{remark}
\label{r1}If we choose $\alpha =1$ in Theorem 1, we obtain%
\begin{eqnarray*}
\frac{1}{b-a}\int_{a}^{b}f\left( x\right) dx &\leq &\left[ f\left( a\right)
+f\left( b\right) \right] \int_{0}^{1}h(t)dt \\
&\leq &\left[ f\left( a\right) +f\left( b\right) \right] \left(
\int_{0}^{1}\left( h\left( t\right) \right) ^{q}dt\right) ^{\frac{1}{q}}.
\end{eqnarray*}
\end{remark}

\begin{corollary}
(1) If we choose $h\left( \lambda \right) =\lambda $ in Remark \ref{r1}, we
get%
\begin{equation*}
\frac{1}{b-a}\int_{a}^{b}f\left( x\right) dx\leq \frac{f\left( a\right)
+f\left( b\right) }{2}=\frac{f\left( a\right) +f\left( b\right) }{\left(
q+1\right) ^{\frac{1}{q}}}
\end{equation*}%
for ordinary convex functions.

(2) If we choose $h\left( \lambda \right) =1$ in Remark \ref{r1}, we get%
\begin{equation*}
\frac{2}{b-a}\int_{a}^{b}f\left( x\right) dx\leq 2\left( f\left( a\right)
+f\left( b\right) \right)
\end{equation*}%
for $P-$functions. This inequality is refinement of right hand side of (\ref%
{a}) for $P-$functions.

(3) If we choose $h\left( \lambda \right) =$ $\lambda ^{s}$ in Remark \ref%
{r1}, we get%
\begin{equation*}
\frac{1}{b-a}\int_{0}^{1}f\left( x\right) dx\leq \frac{f\left( a\right)
+f\left( b\right) }{s+1}\leq \frac{f\left( a\right) +f\left( b\right) }{%
\left( sq+1\right) ^{\frac{1}{q}}}
\end{equation*}%
for $s-$convex functions in the second sense with $s\in \left( 0,1\right] $%
.\bigskip
\end{corollary}

\begin{theorem}
\label{metu2} Let $f\in SX(h,I),$ $a,b\in I$ with $a<b,$ $h$ is
superadditive on $I$ and $f\in L_{1}[a,b],$ $h\in L_{1}[0,1].$ Then one has
inequality for $h-$convex functions via fractional integrals%
\begin{equation}
\frac{\Gamma (\alpha )}{(b-a)^{\alpha }}\left[ J_{a^{+}}^{\alpha
}(b)+J_{b^{-}}^{\alpha }(a)\right] \leq \frac{h(1)}{\alpha }\left[ f(a)+f(b)%
\right] .  \label{xy}
\end{equation}
\end{theorem}

\begin{proof}
Since $f\in SX(h,I)$, we have%
\begin{equation*}
f(tx+(1-t)y)\leq h(t)f(x)+h(1-t)f(y)
\end{equation*}%
and%
\begin{equation*}
f((1-t)x+ty)\leq h(1-t)f(x)+h(t)f(y).
\end{equation*}%
By adding these inequalities we get%
\begin{equation}
f(tx+(1-t)y)+f((1-t)x+ty)\leq \left[ h(t)+h(1-t)\right] \left[ f(x)+f(y)%
\right] .  \label{x}
\end{equation}

By using (\ref{x}) with $x=a$ and $y=b$ and $h$ is superadditive, we get%
\begin{equation}
f(ta+(1-t)b)+f((1-t)a+tb)\leq h(1)\left[ f(a)+f(b)\right] .  \label{y}
\end{equation}%
Then multiplying both sides of (\ref{y}) by $t^{\alpha -1}$ and integrating
the resulting inequality with respect to $t$ over $[0,1]$, we get%
\begin{eqnarray*}
&&\int_{0}^{1}t^{\alpha -1}\left[ f(ta+(1-t)b)+f((1-t)a+tb)\right] dt \\
&\leq &\int_{0}^{1}t^{\alpha -1}h(1)\left[ f(a)+f(b)\right] dt,
\end{eqnarray*}%
\begin{eqnarray*}
&&\frac{\Gamma (\alpha )}{(b-a)^{\alpha }}\left[ J_{a^{+}}^{\alpha
}(b)+J_{b^{-}}^{\alpha }(a)\right] \\
&\leq &h(1)\left[ f(a)+f(b)\right] \int_{0}^{1}t^{\alpha -1}dt.
\end{eqnarray*}%
This completes the proof.
\end{proof}

\begin{remark}
\label{r2}If we choose $\alpha =1$ in Theorem \ref{metu2}, then (\ref{xy})
reduce to special version of right hand side of (\ref{10}).
\end{remark}

\begin{theorem}
Let $h:J\subset 
\mathbb{R}
\rightarrow 
\mathbb{R}
$ and $f:\left[ a,b\right] \rightarrow 
\mathbb{R}
$ be positive functions with $0\leq a<b$ and $h^{q}\in L_{1}\left[ 0,1\right]
,$ $f\in L_{1}\left[ a,b\right] .$ If $\left\vert f^{\prime }\right\vert $
is an $h-$convex mapping on $\left[ a,b\right] $, then the following
inequality for fractional integrals holds,%
\begin{eqnarray}
&&\left\vert \frac{f\left( a\right) +f\left( b\right) }{2}-\frac{\Gamma
\left( \alpha +1\right) }{2\left( b-a\right) ^{\alpha }}\left[
J_{a+}^{\alpha }f\left( b\right) +J_{b-}^{\alpha }f\left( a\right) \right]
\right\vert  \label{mt} \\
&\leq &\frac{\left( b-a\right) \left[ \left\vert f^{\prime }\left( a\right)
\right\vert +\left\vert f^{\prime }\left( b\right) \right\vert \right] }{2}%
\left[ \left( \frac{2^{\alpha p+1}-1}{2^{\alpha p+1}\left( \alpha p+1\right) 
}\right) ^{\frac{1}{p}}-\left( \frac{1}{2^{\alpha p+1}\left( \alpha
p+1\right) }\right) ^{\frac{1}{p}}\right]  \notag \\
&&\times \left[ \left( \int_{0}^{\frac{1}{2}}\left( h\left( t\right) \right)
^{q}dt\right) ^{\frac{1}{q}}+\left( \int_{\frac{1}{2}}^{1}\left( h\left(
t\right) \right) ^{q}dt\right) ^{\frac{1}{q}}\right]  \notag
\end{eqnarray}%
where $\alpha >0$, $p>1$ and $p^{-1}+q^{-1}=1.$
\end{theorem}

\begin{proof}
From Lemma \ref{l1} and using the properties of modulus, we have%
\begin{eqnarray*}
&&\left\vert \frac{f\left( a\right) +f\left( b\right) }{2}-\frac{\Gamma
\left( \alpha +1\right) }{2\left( b-a\right) ^{\alpha }}\left[
J_{a+}^{\alpha }f\left( b\right) +J_{b-}^{\alpha }f\left( a\right) \right]
\right\vert \\
&\leq &\frac{b-a}{2}\int_{0}^{1}\left\vert \left( 1-t\right) ^{\alpha
}-t^{\alpha }\right\vert \left\vert f^{\prime }\left( ta+\left( 1-t\right)
b\right) \right\vert dt.
\end{eqnarray*}%
Since $\left\vert f^{\prime }\right\vert $ is $h-$convex on $\left[ a,b%
\right] $, we have%
\begin{eqnarray}
&&\left\vert \frac{f\left( a\right) +f\left( b\right) }{2}-\frac{\Gamma
\left( \alpha +1\right) }{2\left( b-a\right) ^{\alpha }}\left[
J_{a+}^{\alpha }\left( b\right) +J_{b-}^{\alpha }\left( a\right) \right]
\right\vert  \label{m2} \\
&\leq &\frac{b-a}{2}\left\{ \int_{0}^{\frac{1}{2}}\left[ \left( 1-t\right)
^{\alpha }-t^{\alpha }\right] \left[ h\left( t\right) \left\vert f^{\prime
}\left( a\right) \right\vert +h\left( 1-t\right) \left\vert f^{\prime
}\left( b\right) \right\vert \right] dt\right.  \notag \\
&&+\left. \int_{\frac{1}{2}}^{1}\left[ t^{\alpha }-\left( 1-t\right)
^{\alpha }\right] \left[ h\left( t\right) \left\vert f^{\prime }\left(
a\right) \right\vert +h\left( 1-t\right) \left\vert f^{\prime }\left(
b\right) \right\vert \right] dt\right\}  \notag \\
&=&\frac{b-a}{2}\left\{ \left\vert f^{\prime }\left( a\right) \right\vert
\int_{0}^{\frac{1}{2}}\left( 1-t\right) ^{\alpha }h\left( t\right)
dt-\left\vert f^{\prime }\left( a\right) \right\vert \int_{0}^{\frac{1}{2}%
}t^{\alpha }h\left( t\right) dt\right.  \notag \\
&&+\left. \left\vert f^{\prime }\left( b\right) \right\vert \int_{0}^{\frac{1%
}{2}}\left( 1-t\right) ^{\alpha }h\left( 1-t\right) dt-\left\vert f^{\prime
}\left( b\right) \right\vert \int_{0}^{\frac{1}{2}}t^{\alpha }h\left(
1-t\right) dt\right.  \notag \\
&&+\left. \left\vert f^{\prime }\left( a\right) \right\vert \int_{\frac{1}{2}%
}^{1}t^{\alpha }h\left( t\right) dt-\left\vert f^{\prime }\left( a\right)
\right\vert \int_{\frac{1}{2}}^{1}\left( 1-t\right) ^{\alpha }h\left(
t\right) dt\right.  \notag \\
&&+\left. \left\vert f^{\prime }\left( b\right) \right\vert \int_{\frac{1}{2}%
}^{1}t^{\alpha }h\left( 1-t\right) dt-\left\vert f^{\prime }\left( b\right)
\right\vert \int_{\frac{1}{2}}^{1}\left( 1-t\right) ^{\alpha }h\left(
1-t\right) dt\right\} .  \notag
\end{eqnarray}%
In the right hand side of above inequality by using H\"{o}lder inequality
for $p^{-1}+q^{-1}=1$ and $p>1,$ we get%
\begin{equation*}
\int_{0}^{\frac{1}{2}}\left( 1-t\right) ^{\alpha }h\left( t\right) dt=\int_{%
\frac{1}{2}}^{1}t^{\alpha }h\left( 1-t\right) dt\leq \left[ \frac{2^{\alpha
p+1}-1}{2^{\alpha p+1}\left( \alpha p+1\right) }\right] ^{\frac{1}{p}}\left(
\int_{0}^{\frac{1}{2}}\left[ h\left( t\right) \right] ^{q}dt\right) ^{\frac{1%
}{q}},
\end{equation*}%
\begin{equation*}
\int_{0}^{\frac{1}{2}}\left( 1-t\right) ^{\alpha }h\left( 1-t\right)
dt=\int_{\frac{1}{2}}^{1}t^{\alpha }h\left( t\right) dt\leq \left[ \frac{%
2^{\alpha p+1}-1}{2^{\alpha p+1}\left( \alpha p+1\right) }\right] ^{\frac{1}{%
p}}\left( \int_{\frac{1}{2}}^{1}\left[ h\left( t\right) \right]
^{q}dt\right) ^{\frac{1}{q}},
\end{equation*}%
\begin{equation*}
\int_{0}^{\frac{1}{2}}t^{\alpha }h\left( t\right) dt=\int_{\frac{1}{2}%
}^{1}\left( 1-t\right) ^{\alpha }h\left( 1-t\right) dt\leq \left[ \frac{1}{%
2^{\alpha p+1}\left( \alpha p+1\right) }\right] ^{\frac{1}{p}}\left(
\int_{0}^{\frac{1}{2}}\left[ h\left( t\right) \right] ^{q}dt\right) ^{\frac{1%
}{q}}
\end{equation*}%
and%
\begin{equation*}
\int_{0}^{\frac{1}{2}}t^{\alpha }h\left( 1-t\right) dt=\int_{\frac{1}{2}%
}^{1}\left( 1-t\right) ^{\alpha }h\left( t\right) dt\leq \left[ \frac{1}{%
2^{\alpha p+1}\left( \alpha p+1\right) }\right] ^{\frac{1}{p}}\left( \int_{%
\frac{1}{2}}^{1}\left[ h\left( t\right) \right] ^{q}dt\right) ^{\frac{1}{q}}.
\end{equation*}%
Then using the above inequalities in the right hand side of (\ref{m2}), we
get 
\begin{eqnarray*}
&&\left\vert \frac{f\left( a\right) +f\left( b\right) }{2}-\frac{\Gamma
\left( \alpha +1\right) }{2\left( b-a\right) ^{\alpha }}\left[
J_{a+}^{\alpha }\left( b\right) +J_{b-}^{\alpha }\left( a\right) \right]
\right\vert \\
&\leq &\frac{b-a}{2}\left\{ \left\vert f^{\prime }\left( a\right)
\right\vert \left\{ \left( \left[ \frac{2^{\alpha p+1}-1}{2^{\alpha
p+1}\left( \alpha p+1\right) }\right] ^{\frac{1}{p}}-\left[ \frac{1}{%
2^{\alpha p+1}\left( \alpha p+1\right) }\right] ^{\frac{1}{p}}\right) \left(
\int_{0}^{\frac{1}{2}}\left[ h\left( t\right) \right] ^{q}dt\right) ^{\frac{1%
}{q}}\right\} \right. \\
&&\left. +\left( \left[ \frac{2^{\alpha p+1}-1}{2^{\alpha p+1}\left( \alpha
p+1\right) }\right] ^{\frac{1}{p}}-\left[ \frac{1}{2^{\alpha p+1}\left(
\alpha p+1\right) }\right] ^{\frac{1}{p}}\right) \left( \int_{\frac{1}{2}%
}^{1}\left[ h\left( t\right) \right] ^{q}dt\right) ^{\frac{1}{q}}\right\} \\
&&+\left\vert f^{\prime }\left( b\right) \right\vert \left\{ \left( \left[ 
\frac{2^{\alpha p+1}-1}{2^{\alpha p+1}\left( \alpha p+1\right) }\right] ^{%
\frac{1}{p}}-\left[ \frac{1}{2^{\alpha p+1}\left( \alpha p+1\right) }\right]
^{\frac{1}{p}}\right) \left( \int_{\frac{1}{2}}^{1}\left[ h\left( t\right) %
\right] ^{q}dt\right) ^{\frac{1}{q}}\right. \\
&&\left. +\left( \left[ \frac{2^{\alpha p+1}-1}{2^{\alpha p+1}\left( \alpha
p+1\right) }\right] ^{\frac{1}{p}}-\left[ \frac{1}{2^{\alpha p+1}\left(
\alpha p+1\right) }\right] ^{\frac{1}{p}}\right) \left( \int_{0}^{\frac{1}{2}%
}\left[ h\left( t\right) \right] ^{q}dt\right) ^{\frac{1}{q}}\right\} \\
&=&\frac{b-a}{2}\left\{ \left\vert f^{\prime }\left( a\right) \right\vert
\left( \left[ \frac{2^{\alpha p+1}-1}{2^{\alpha p+1}\left( \alpha p+1\right) 
}\right] ^{\frac{1}{p}}-\left[ \frac{1}{2^{\alpha p+1}\left( \alpha
p+1\right) }\right] ^{\frac{1}{p}}\right) \left[ \left( \int_{0}^{\frac{1}{2}%
}\left[ h\left( t\right) \right] ^{q}dt\right) ^{\frac{1}{q}}+\left( \int_{%
\frac{1}{2}}^{1}\left[ h\left( t\right) \right] ^{q}dt\right) ^{\frac{1}{q}}%
\right] \right. \\
&&\left. +\left\vert f^{\prime }\left( b\right) \right\vert \left( \left[ 
\frac{2^{\alpha p+1}-1}{2^{\alpha p+1}\left( \alpha p+1\right) }\right] ^{%
\frac{1}{p}}-\left[ \frac{1}{2^{\alpha p+1}\left( \alpha p+1\right) }\right]
^{\frac{1}{p}}\right) \left[ \left( \int_{0}^{\frac{1}{2}}\left[ h\left(
t\right) \right] ^{q}dt\right) ^{\frac{1}{q}}+\left( \int_{\frac{1}{2}}^{1}%
\left[ h\left( t\right) \right] ^{q}dt\right) ^{\frac{1}{q}}\right] \right\}
\\
&=&\frac{\left( b-a\right) \left[ \left\vert f^{\prime }\left( a\right)
\right\vert +\left\vert f^{\prime }\left( b\right) \right\vert \right] }{2}%
\left[ \left( \frac{2^{\alpha p+1}-1}{2^{\alpha p+1}\left( \alpha p+1\right) 
}\right) ^{\frac{1}{p}}-\left( \frac{1}{2^{\alpha p+1}\left( \alpha
p+1\right) }\right) ^{\frac{1}{p}}\right] \\
&&\times \left[ \left( \int_{0}^{\frac{1}{2}}\left( h\left( t\right) \right)
^{q}dt\right) ^{\frac{1}{q}}+\left( \int_{\frac{1}{2}}^{1}\left( h\left(
t\right) \right) ^{q}dt\right) ^{\frac{1}{q}}\right]
\end{eqnarray*}%
which is the desired result. The proof is completed.
\end{proof}

\end{document}